\documentclass[reqno,12pt]{amsart}
\usepackage{amsmath}
\usepackage{amssymb}
\usepackage{amsfonts}
\usepackage{graphicx}
\usepackage{natbib}
\usepackage{hyperref}
\usepackage{mathrsfs}
\usepackage[english]{babel}
\usepackage[autostyle]{csquotes}
\usepackage{amsaddr}
\usepackage{mathabx}

\newtheorem{theorem}{Theorem}[section]

\theoremstyle{definition}
\newtheorem{definition}[theorem]{Definition}
\theoremstyle{corollary}
\newtheorem{corollary}[theorem]{Corollary}
\theoremstyle{example}

\theoremstyle{note}

\theoremstyle{notation}

\numberwithin{equation}{section}

\bibfont{\footnotesize}

\begin{document}
\title[C-image partition regularity near zero]
{C-image partition regularity near zero}

\author{Sourav Kanti Patra}
\address{Department of Mathematics, Ramakrishna Mission Vidyamandira, Belur Math, Howrah-711202, West Bengal, India}
\email{souravkantipatra@gmail.com}

\author{Sukrit Chakraborty}
\address{Statistics and Mathematics Unit, Indian Statistical Institute, Kolkata-700108, West Bengal, India}
\email{sukritpapai@gmail.com}


\newcommand{\acr}{\newline\indent}
\maketitle

\begin{abstract}
The concept of image partition regularity near zero was first investigated by De and Hindman. In contrast to the finite case , infinite image partition regular matrices near zero are very fascinating to analyze. In this regard the abstraction of centrally image partition regular matrices near zero was introduced by Biswas, De and Paul. In this paper we propose the notion of matrices that are C-image partition regular near zero for dense subsemigropus of $((0,\infty),+)$.

AMS subjclass [2010] : Primary : 05D10 Secondary : 22A15
\end{abstract}

\maketitle

\section{Introduction}
A finite or infinite matrix $A$, with entries from $\mathbb{Q}$, is image partition regular provided that whenever $\mathbb{N}$ is finitely colored, there must be some $\overset{\rightarrow}{x}$ with entries from $\mathbb{N}$ such that all entries of $A \overset{\rightarrow}{x}$ are in the same color class. Several characterizations of infinite image partition regular matrices involve the notion of \enquote{first entries matrix}, a concept based on Deuber's $(m,p,c)$ sets. For an overview on this topic, the reader is referred to \cite{MR0325406, MR2016182}. We give the definition below.
\begin{definition}
Let $A$ be a $u \times v$ matrix with rational entries. Then $A$ is a first entries matrix if and only if no row of $A$ is $\overset{\rightarrow}{0}$ and there exist $d_1, d_2, \cdots , d_v \in \{x \in \mathbb{Q} : x > 0\}$ such that, whenever $i \in \{1, 2, \cdots , v\}$ and $l = \min\{j \in \{1, 2, \cdots , v\} : a_{i,j} \neq 0 \}$, then $d_l$ is a first entry of $A$.
\end{definition} 
It is well known that for finite matrices, image partition regularity behaves well with respect to central subsets of the underlying semigroup. central sets were introduced in \cite{frusbook} and were defined in terms of notions of topological dynamics. These sets enjoy very strong combinatorics properties. (see for further details Proposition $8.21$ of \cite{frusbook}, Chapter $14$ of \cite{hindman1998algebra}.) They have a nice characterization in terms of the algebraic structure of $\beta \mathbb{N}$, the Stone-{\v{C}}ech compactification of $\mathbb{N}$. We shall present this characterization below, after introducing the necessary background information.

Let $(S,+)$ be an infinite discrete semigroup. We take the points of $\beta S$ to be all the ultrafilters on $S$, the principal ultrafilters being identified with points of $S$. Given a set $A \subseteq S$, $\overline{A} = \{p \in \beta S : A \in p\}$. the set $\{\overline{A} : A \subseteq S\}$ is a basis for the open sets (as well as the closed sets) of $\beta S$. There is a natural extension of the operation $+$ on $S$ to $\beta S$ making $\beta S$ a compact right topological semigroup (meaning that for any $p \in \beta S$, the function $\rho_p : \beta S \to \beta S$ defined by $\rho_p(q)= q+p$ is continuous) with $S$ contained in its topological center (meaning that for any $x\in S$, the function $\lambda_x : \beta S \to \beta S$ defined by $\lambda_x(q) = x+q$ is continuous). Given $p,q \in \beta S$ and $A \subseteq S$, $A \in p+q$ if and only if $\{x \in S : -x + A \in q\} \in p$, where $-x + A = \{y \in S : x+y \in A\}$.

A nonempty subset $I$ of a semigroup $(T,+)$ is called a left ideal of $T$ if $T + I \subseteq I$, a right ideal if $I + T \subseteq I$, and a two sided ideal (or simply an ideal) if it is both a left and a right ideal. A minimal left ideal is a left ideal that does not contain any proper left ideal. Similarly, we can define a minimal right ideal and the smallest ideal.

Any compact Hausdorff right topological semigroup $(T,+)$ contains idempotents and also has the unique smallest two sided ideal 
\begin{align*}
K(T) &= \bigcup\{\text{L : L is a minimal left ideal of T}\}\\
&= \bigcup\{\text{R : R is a minimal right ideal of T}\}. 
\end{align*}  
Given a minimal left ideal $L$ and a minimal right ideal $R$, it easily follows that $L \cap R$ is a group and thus in particular contains an idempotent. If $p$ and $q$ are idempotents in $T$, we write $p \leqslant q$ if and only if $p + q = q + p =p$. An idempotent is minimal with respect to this relation if and only if it is a member of the smallest ideal $K(T)$ of $T$.

See \cite{hindman1998algebra} for an elementary introduction to the algebra of $\beta S$ and for any unfamiliar details.
\begin{definition}
Let $(S,+)$ be an infinite discrete semigroup. A set $C \subseteq S$ is central if and only if there is some minimal idempotent $p$ in $(\beta S, +)$ such that $C \in p$. $C$ is called a central$^*$ set if it belongs to every minimal idempotent of $(\beta S, +)$.
\end{definition} 
We will be considering the sets which are dense in $((0, \infty), +)$. Here \enquote{dense} means with respect to the usual topology on $((0, \infty), +)$. When passing through the Stone-{\v{C}}ech compactification of such a semigroup $S$, we deal with the set $S$ with the discrete topology.
\begin{definition}
If $S$ is a dense subsemigroup of $((0,\infty),+)$, one defines, $O^+(S) = \{p \in \beta S : (0,\epsilon) \cap S \in p \text{ for all }\epsilon > 0\}$. 
\end{definition}
It is proved in Lemma $2.5$ of \cite{hindman1999semigroup}, that $O^+(S)$ is a compact right topological semigroup of $(\beta S, +)$. It was also noted there that $O^+(S)$ is disjoint from $K(\beta S)$, and hence gives some new information which is not available from $K(\beta S)$. Being a compact right topological semigroup $O^+(S)$ contains minimal idempotents. We denote $K(O^+(S))$ to be the smallest ideal of $O^+(S)$. Note that idempotents of $K(O^+(S))$ are minimal idempotents of $O^+(S)$.
\begin{definition}
Let $S$ be a dense subsemigroup of $((0,\infty),+)$. A set $C \subseteq S$ is central near zero if and only if there is some minimal idempotent $p$ in $O^+(S)$ such that $C \in p$. $C$ is central$^*$ set near zero if it belongs to every minimal idempotent of $O^+(S)$. 
\end{definition}
In \cite{11de2012}, nice combinatorial algebraic properties of central sets near zero had been established. Now we present some well known characterization of image partition regularity of matrices following Theorem $2.10$ of \cite{5hind2002}. 
\begin{theorem} \label{thm0.0}
Let $A$ be a $p \times q$ matrix with rational entries for some $p, q \in \mathbb{N}$. The following statements are equivalent:
\begin{enumerate}
\item $A$ is an image partition regular matrix.
\item Let $C$ be a central subset of $\mathbb{N}$. Then $A\overset{\rightarrow}{x} \in C^p$ for some $\overset{\rightarrow}{x} \in \mathbb{N}^q$.
\item Let $C$ be a central subset of $\mathbb{N}$. Then $\{\overset{\rightarrow}{x} \in \mathbb{N}^q : A \overset{\rightarrow}{x} \in C^p\}$ is also a central subset $\mathbb{N}^q$.
\item For some $m\in \mathbb{N}$ there exist a $q \times m$ matrix $X$ with entries from $\omega$ and no row equal to $\overset{\rightarrow}{0}$. There also exists a $p \times m$ first entries matrix $Y$ with entries from $\omega$. Then for some $n \in \mathbb{N}$ one has $n$ is the only first entry of $Y$ and $AX = Y$.
\item For some $m\in \mathbb{N}$ there exist a $p \times m$ first entries matrix $Z$ with all entries from $\omega$. Then for some $c \in \mathbb{N}$ one has $c$ as the only first entry of $Z$ and for each $\overset{\rightarrow}{y} \in \mathbb{N}^m$ there exists $\overset{\rightarrow}{x} \in \mathbb{N}^q$ such that $A\overset{\rightarrow}{x} = Z\overset{\rightarrow}{y}$.
\item \label{thm0.06}For some $m\in \mathbb{N}$ there exist a $p \times m$ first entries matrix $Z$ such that for each $\overset{\rightarrow}{y} \in \mathbb{N}^m$ one has $A\overset{\rightarrow}{x} = Z\overset{\rightarrow}{y}$ for some $\overset{\rightarrow}{x}  \in \mathbb{N}^q$. 
\item For each $\overset{\rightarrow}{s} \in \mathbb{Q} \setminus \{\overset{\rightarrow}{0}\}$ one gets $\begin{pmatrix}
h\overset{\rightarrow}{s}\\ M
\end{pmatrix}$ to be image partition regular for some $h \in \mathbb{Q} \setminus \{0\}$.
\item For any $m \in \mathbb{N}$ let $\phi_1, \phi_2, \cdots, \phi_m$ be non zero linear mappings from $\mathbb{Q}^q$ to $\mathbb{Q}$. Let $C$ be central in $\mathbb{N}$. Then one has $A\overset{\rightarrow}{x} \in C^p$ for some $\overset{\rightarrow}{x} \in \mathbb{N}^q$ and for each $i \in \{1,2, \cdots ,m\}$, $\phi_i(\overset{\rightarrow}{x})\neq 0$.
\item Let $C$ be any central set in $\mathbb{N}$. Then there exists $\overset{\rightarrow}{x} \in \mathbb{N}^p$ such that $\overset{\rightarrow}{y} = A\overset{\rightarrow}{x} \in C^p$, all entries of $\overset{\rightarrow}{x}$ are distinct and for all $i,j \in \{1,2, \cdots , u\}$, if row $i$ and $j$ of $A$ are unequal, then $y_i \neq y_j$. 
\end{enumerate}
\end{theorem}
In \cite{6hind2003}, the authors presented some contrast between finite and infinite image partition regular matrices and showed that some of the interesting properties of finite image partition regular matrices could not be generalized for infinite image partiton regular matrices. In this regard the notion of centrally image partition regular matrices were introduced in Definition $2.7$ of \cite{6hind2003}.
\begin{definition}\label{def1.6}
Let $M$ be an $\omega \times \omega$ matrix with rational entries. Then $M$ is centrally image partition regular if and only if for any central set $C$ in $\mathbb{N}$ one has $M\overset{\rightarrow}{x} \in C^{\omega}$ for some $\overset{\rightarrow}{x} \in \mathbb{N}^{\omega}$. It is noteworthy to mention that $\overset{\rightarrow}{x}$ may depend on the choice of $C$.
\end{definition} 
Note that Definition \ref{def1.6} has a natural generalization for an arbitrary subsemigroup $S$ of $((0,\infty),+)$. In \cite{biswaspaul}, the authors introduced another natural candidate to generalize the properties of finite image partition regularity near zero to the case of infinite matrices.

We now recall Definitions $1.7$ and $1.8$ of \cite{biswaspaul} respectively. 
\begin{definition}
Let $S$ be a dense subsemigroup of $((0,\infty),+)$. Let $p,q \in \mathbb{N}$ and $A$ be a $p \times q$ matrix with rational entries. The matrix $A$ is image partition regular near zero over $S$ if and only if whenever $r \in \mathbb{N}$, $\epsilon > 0$ and $S = \bigcup_{i=1}^{r}C_i$, one has $A\overset{\rightarrow}{x} \in (C \cap (0,\epsilon))^p$ for some $i \in \{1,2, \cdots ,r\}$ and $\overset{\rightarrow}{x} \in \mathbb{N}^q$. 
\end{definition}
\begin{definition}\label{def1.8}
Let $S$ be a dense subsemigroup $((0,\infty), +)$. Let $A$ be an $\omega \times \omega$ matrix with entries from $\mathbb{Q}$. $A$ is centrally image partition regular near zero over $S$ if for a central set $C$ near zero in $S$, one has $A\overset{\rightarrow}{x} \in C^{\omega}$ for some $\overset{\rightarrow}{x} \in S^{\omega}$.
\end{definition}
In Section \ref{sec2}, we shall introduce the concept of C-image partition regular matrices near zero, which is an interesting subclass of centrally image partition regular matrices near zero. We shall see that both these image partition regularities behave almost the same.

In Section \ref{sec3}, we will give some examples of C-image partition regular matrices.
\section{C-image partition regular matrices near zero} \label{sec2}
In this section we shall define C-image partition regularity near zero for dense subsemigroup $S$ of $((0,\infty),+)$. Let us recall Definitions $3.1$ and $3.2$ of \cite{9bayat2016} respectively.
\begin{definition}
Let $S$ be a dense subsemigroup of $(0,\infty)$. Let $\tau_0$ be the set of all sequences in $S$ which converge to zero.
\end{definition}
\begin{definition}
Let $S$ be a dense subsemigroup of $(0,\infty)$ and $A \subseteq S$. Then $A$ is J-set near zero if and only if whenever $F \in \mathcal{P}_f(\tau_0)$ and $\delta > 0$, there exists $a \in S\cap (0,\delta)$ and $H \in \mathcal{P}_f(\mathbb{N})$ such that $f \in F$, $a + \sum_{t \in H}f(t) \in A$. 
\end{definition}
We now present the central sets theorem near zero.
\begin{theorem}
Let $S$ be a dense subsemigroup of $((0,\infty),+)$. Let $A$ be a central subset of $S$ near zero. Then for each $\delta \in (0,1)$, there exists functions $\alpha_{\delta} : \mathcal{P}_f(\tau_{0}) \to S$ and $H_{\delta} : \mathcal{P}_f(\tau_{0}) \to \mathcal{P}_f(\mathbb{N})$ such that 
\begin{enumerate}
\item $\alpha_{\delta}(F) < \delta$ for each $F \in \mathcal{P}_f(\tau_0)$. 
\item If $F,G \in \mathcal{P}_f(\tau_0)$ and $F \subseteq G$, then $\max H_{\delta}(F) < \min H_{\delta(G)}$ and 
\item whenever $m \in \mathbb{N}$, $G_1,G_2, \cdots ,G_m \in \mathcal{P}_f(\tau_0)$, $G_1 \subseteq G_2 \subseteq \cdots \subseteq G_m$ and for each $i \in \{1,2,\cdots,m\}$, $f_i \in G_i$ one has $$\sum_{i=1}^{m}\Big(\alpha_{\delta}(G_i)+\sum_{t \in H_{\delta}(G_i)}f_i(t)\Big) \in A.$$ 
\end{enumerate}
\end{theorem}
\begin{proof}
See Theorem 3.5 of \citep{9bayat2016}. 
\end{proof}
For a dense subsemigroup $S$ of $((0,\infty),+)$, a set $A \subseteq S$ is said to be a $C$-set near zero if it satisfy the conclusion of the central Sets Theorem near zero. \\
So we have the following definition which is Definition 3.6(a)  of \citep{9bayat2016}. 
\begin{definition}
Let $S$ be a dense subsemigroup of $((0,\infty),+)$ and let $A \subseteq S$. We say that $A$ is a $C$-set near zero if and only if for each $\delta \in (0,1)$, there exist functions $\alpha_{\delta} : \mathcal{P}_f(\tau_{0}) \to S$ and $H_{\delta} : \mathcal{P}_f(\tau_{0}) \to \mathcal{P}_f(\mathbb{N})$ such that 
\begin{enumerate}
\item $\alpha_{\delta}(F) < \delta$ for each $F \in \mathcal{P}_f(\tau_0)$. 
\item If $F,G \in \mathcal{P}_f(\tau_0)$ and $F \subseteq G$, then $\max H_{\delta}(F) < \min H_{\delta(G)}$ and 
\item whenever $m \in \mathbb{N}$, $G_1,G_2, \cdots ,G_m \in \mathcal{P}_f(\tau_0)$, $G_1 \subseteq G_2 \subseteq \cdots \subseteq G_m$ and for each $i \in \{1,2,\cdots,m\}$, $f_i \in G_i$ one has $$\sum_{i=1}^{m}\Big(\alpha_{\delta}(G_i)+\sum_{t \in H_{\delta}(G_i)}f_i(t)\Big) \in A.$$ 
\end{enumerate} 
\end{definition}
The following definition is Definition 3.6(b) of \citep{9bayat2016}. 
\begin{definition}
Let $S$ be a dense subsemigroup of $((0,\infty),+)$. Define $J_0(S) = \{p \in O^+(S) : \text{ for all } A \in p, \text{ }A \text{ is a }J\text{-set near zero}\}$. 
\end{definition}
\begin{theorem}\label{thm2. 12}
Let $S$ be a dense subsemigroup of $((0,\infty),+)$. Then $J_0(S)$ is a compact two-sided ideal of $\beta S$. 
\end{theorem}
\begin{proof}
See Theorem 3.9 of \citep{9bayat2016}.
\end{proof}
\begin{theorem}
Let $S$ be a dense subsemigroup of $((0,\infty),+)$ and $A \subseteq S$. Then $A$ is a C-set near zero if and only if there is an idempotent $p \in \overline{A} \cap J_0(S)$.
\end{theorem}
\begin{proof}
See Theorem $3.14$ of \cite{frusbook}.
\end{proof}
We call a set $A \subseteq S$ to be C$^*$-set near zero if and only if it is a member of every idempotent in $J_0(S)$.
\begin{theorem}\label{thm2.8}
Let $S$ be a dense subsemigroup of $((0,\infty),+)$, let
$u, v \in \mathbb{N}$, and let $M$ be a $u \times v$ matrix with entries from $\omega$ which satisfies the first
entries condition. Let $A$ be C-set near zero in $S$. If, for every first entry $c$ of $M$, $cS$ is a C$^*$-set near zero, then there exist sequences $\langle x_{1,n} \rangle_{n=1}^{\infty}, \langle x_{2,n} \rangle_{n=1}^{\infty}, \cdots \langle x_{v,n} \rangle_{n=1}^{\infty}$ in $\tau_0$ such that for every $F \in \mathcal{P}_f(\mathbb{N})$, $\overset{\rightarrow}{x}_F \in (S \setminus \{0\})^v$ and $M\overset{\rightarrow}{x}_F \in A^u$, where
$$\overset{\rightarrow}{x}_F = \begin{pmatrix}
\sum_{x\in F}x_{1,n}\\ \sum_{x\in F}x_{2,n}\\ \vdots \\ \sum_{x\in F}x_{v,n}\\ 
\end{pmatrix}.$$
\end{theorem}
\begin{proof}
The proof is almost same as that of Theorem $15.5$ of \cite{hindman1998algebra}.
\end{proof}
As a consequence of the above theorem, we have the following corollary.
\begin{corollary}\label{cor2.9}
Let $S$ be a dense subsemigroup of $((0,\infty),+)$ for which $cS$ is a C$^*$-set near zero for each $c\in \mathbb{N}$. Let $u, v \in \mathbb{N}$ and $M$ be a $u \times v$ matrix with entries from $\mathbb{Q}$ which is image partition regular over $\mathbb{N}$. Then for each C-set near zero $A$, there exists $\overset{\rightarrow}{x} \in S^v$ such that $M\overset{\rightarrow}{x} \in A^u$.
\end{corollary}
We shall now introduce the notion of C-image partition regular matrix near zero.
\begin{definition} \label{def2.10}
Let $A$ be an $\omega \times \omega$ matrix with entries from $\mathbb{Q}$. The matrix $A$ is C-image partition regular near zero if and only if for every C-set near zero $C$ of $\mathbb{N}$ there exists $\overset{\rightarrow}{x} \in \mathbb{N}^{\omega}$ such that $A\overset{\rightarrow}{x} \in C^{\omega}$.
\end{definition}
From Definitions \ref{def2.10} and \ref{def1.8}, it is clear that every C-image partition regular matrix near zero is centrally image partition regular.
In the following theorem we shall see that C-image partition regular matrices near zero are closed under diagonal sums.
\begin{theorem}
Let $S$ be a dense subsemigroup of $((0,\infty),+)$. For each $n \in \mathbb{N}$, let the matrices $M_n$ be C-image partition regular near zero. Then the matrix $$M = \begin{pmatrix}
M_1 & 0 & 0 & \cdots \\
0 & M_2 & 0 & \cdots \\
0 & 0 & M_3 & \cdots \\
\vdots & \vdots & \vdots & \ddots \\
\end{pmatrix}$$ is also C-image partition regular near zero. 
\end{theorem}
\begin{proof}
Let $A$ be a C-set near zero. For each $n \in \mathbb{N}$, $M_n$ is C-image partition regular. Therefore choose $\overset{\rightarrow}{x}^{(n)} \in S^{\omega}$ such that $\overset{\rightarrow}{y}^{(n)} = M_n \overset{\rightarrow}{x}^{(n)} \in A^{\omega}$ for each $n \in \mathbb{N}$ (by definition). Let $$\overset{\rightarrow}{z} = \begin{pmatrix}
\overset{\rightarrow}{x}^{(1)} \\
\overset{\rightarrow}{x}^{(2)} \\
\vdots \\
\end{pmatrix}.$$ Then all entries of $M \overset{\rightarrow}{z}$ are in $A$.
\end{proof}
\section{Some classes of infinite matrices that are C-image partition regular near zero} \label{sec3}
We now present a class of image partition regular matrices, called the segmented image partition regular matrices which were first introduced in \cite{7hind2000}. There it was shown that segmented image partition regular matrices are centrally image partition regular.
Recall the Definition $3.2$ of \cite{biswaspaul}.
\begin{definition}
Let $M$ be an $\omega \times \omega$ matrix with entries from $\mathbb{Q}$. Then $M$ is a segmented image partition regular matrix if and only if:
\begin{enumerate}
\item No row of $M$ is $\overset{\rightarrow}{0}$.
\item For each $i \in \omega, \{j \in \omega : a_{i,j} \neq 0\}$ is finite.
\item There is an increasing sequence $\langle \alpha_n \rangle_{n=0}^{\infty}$ in $\omega$ such that $\alpha_0 = 0$ and for each $n \in \omega$, $\{\langle a_{i,\alpha_n}
, a_{i,\alpha_{n}+1}, a_{i,\alpha_{n}+2}, \cdots , a_{i,\alpha_{n+1}-1}\rangle : i \in \omega\} \setminus \{\overset{\rightarrow}{0}\}$ is empty or is the set of rows of a finite image partition regular matrix.
\end{enumerate}
If each of these finite image partition regular matrices is a first entries matrix, then $M$ is a segmented first entries matrix. If also the first nonzero entry of each $\langle a_{i,\alpha_n} , a_{i,\alpha_{n}+1}, a_{i,\alpha_{n}+2}, \cdots , a_{i,\alpha_{n+1}-1}\rangle$, if any, is $1$, then $M$ is a
monic segmented first entries matrix.
\end{definition}
The following theorem is Theorem $3.1$ in \cite{biswaspaul}.
\begin{theorem}
Let $S$ be a dense subsemigroup of $((0, \infty), +)$ for which $cS$ is central$^*$ near zero for every $c \in \mathbb{N}$ and let $M$ be a segmented image partition regular matrix with entries from $\omega$. Then $M$ is centrally image partition regular near zero.
\end{theorem} 
The proof of the following theorem is adapted from the proofs of Theorem $3.2$ of \cite{7hind2000} and Theorem $3.1$ of \cite{biswaspaul}.
\begin{theorem} \label{thm3.3}
Let $S$ be a dense subsemigroup of $((0, \infty), +)$ for which $cS$ is C$^*$ near zero for every $c \in \mathbb{N}$ and let $M$ be a segmented image partition regular matrix with entries from $\omega$. Then $M$ is C-image partition regular near zero.
\end{theorem}
\begin{proof} 
Let us denote the columns of $M$ as $\overset{\rightarrow}{c}_{0}, \overset{\rightarrow}{c}_{1}, \overset{\rightarrow}{c}_{2},\cdots$. Pick $\langle\alpha_{n}\rangle_{n=0}^{\infty}$ according to the definition of a segmented image partition regular matrix. For each $n\in \omega$, let $M_{n}$ be the matrix containing columns $\overset{\rightarrow}{c}_{\alpha_{n}},\overset{\rightarrow}{c}_{\alpha_{n}+1},\cdots , \overset{\rightarrow}{c}_{\alpha_{n+1}-1}$. Then the set of non-zero rows of $M_{n}$ becomes finite and, if the case that the set is nonempty, it becomes the set of rows of a finite image partition regular matrix. Define $B_{n} = (M_{0}$  $M_{1}\ldots M_{n})$.
Now in view of Lemma $2.5$ of \cite{hindman1998algebra} $0^{+}(S)$ is a compact right topological semigroup and thus we can pick a minimal idempotent $p\in 0^{+}(S)$. Let $C\subseteq S$ with the property that $C\in p$. Define $C^{*}=\{x\in C : -x+C\in p\}$. Then $C^{*}\in p$ and, for every $x\in C^{*}$, one has $-x+C^{*}\in p$ by Lemma $4.14$ of \cite{hindman1998algebra}.

Now the set of non-zero rows of $M_{n}$ is finite (suppose nonempty) and  is the set of rows of a finite image partition regular matrix over $\mathbb{N}$ and hence by Theorem 2.3 of \cite{dehind1} it is $IPR/S_{0}$. Then by Corollary \ref{cor2.9}, we can pick $\overset{\rightarrow}{x}^{(0)}\in S^{\alpha_{1}-\alpha_{0}}$ such that, if $\overset{\rightarrow}{y}=M_{0}\overset{\rightarrow}{x}^{(0)}$, then $y_{i}\in C^{*}$ for every $i\in \omega$ and also we have that the $i^{th}$ row of $M_{0}$ is non-zero.

We now make the inductive assumption that, for some $m\in \omega$, we have chosen $\overset{\rightarrow}{x}^{(0)},\overset{\rightarrow}{x}^{(1)},\ldots,\overset{\rightarrow}{x}^{(1)}$ with the property that $\overset{\rightarrow}{x}^{(i)}\in S^{\alpha_{i+1}-\alpha_{i}}$ for each $i\in \{0,1,2,\cdots,m\}$, and, if

$$\overset{\rightarrow}{y}=B_{m}\left(\begin{array}{c}\overset{\rightarrow}{x}^{(0)}\\
\overset{\rightarrow}{x}^{(1)}\\ \vdots \\\overset{\rightarrow}{x}^{(m)}\end{array}\right),$$

then $y_{j}\in C^{*}$ for every $j\in \omega$ and one has that the $j^{th}$ row of $B_{m}$ is non-zero.

Let us take $D=\{j\in \omega$ is the row $j$ of $B_{m+1}$ which is not $\overset{\rightarrow}{0}\}$ and note that for each $j\in \omega, -y_{j}+C^{*}\in p$ (Either $y_{j}=0$ or $y_{j}\in C^{*}$). By Corollary \ref{cor2.9} we can choose $\overset{\rightarrow}{x}^{(m+1)}\in S^{\alpha_{m+2}-\alpha_{m+1}}$ with the property that, whenever $\vec z=M_{m+1}\overset{\rightarrow}{x}^{(m+1)}$, one has $z_{j}\in \bigcap _{t\in D}(-y_{t}+C^{*})$ for each $j\in D$.

In this way we can choose an infinite sequence $\langle \overset{\rightarrow}{x}^{(i)} \rangle_{i\in \omega}$ with the property that, for every $i\in \omega$, $\overset{\rightarrow}{x}^{(i)}\in S^{\alpha_{i+1}-\alpha_{i}}$, and, whenever

$$\overset{\rightarrow}{y}=B_{i}\left(\begin{array}{c}\overset{\rightarrow}{x}^{(0)}\\
\overset{\rightarrow}{x}^{(1)}\\ \vdots \\\overset{\rightarrow}{x}^{(i)}\end{array}\right),$$

holds, one has $y_{j}\in C^{*}$ for each $j\in \omega$, where the $j^{th}$ row of $B_{i}$ is non-zero.

Let us take $$\overset{\rightarrow}{x}=\left(\begin{array}{c}\overset{\rightarrow}{x}^{(0)}\\
\overset{\rightarrow}{x}^{(1)}\\ \overset{\rightarrow}{x}^{(2)}\\\vdots \end{array}\right)$$

and also define $\overset{\rightarrow}{y}=M\overset{\rightarrow}{x}$. We note that, for every $j\in \omega$, there exists $m\in \omega$ with the property that $y_{j}$ becomes the $j^{th}$ entry of

$$B_{i}\left(\begin{array}{c}\overset{\rightarrow}{x}^{(0)}\\
\overset{\rightarrow}{x}^{(1)}\\ \vdots \\\overset{\rightarrow}{x}^{(i)}\end{array}\right)$$

for all $i>m$. Thus all the entries of $\overset{\rightarrow}{y}$ falls in $C^{*}$.
\end{proof}
Now we turn our attention to to the methods of constructing C-image partition regular matrices based on existing ones. 
The proof of the following theorem is adapted from Theorem $4.7$ of \cite{7hind2000}.
\begin{theorem}\label{thm3.4}
Let $S$ be a dense subsemigroup of $((0,\infty),+)$ for which $cS$ is C$^*$-set near zero for each $c \in \mathbb{N}$. Let $M$ be a C-image partition regular matrix near zero over $S$ and let $\langle b_n \rangle_{n=1}^{\infty}$ be a sequence in $\mathbb{N}$. Let $$N = \begin{pmatrix} b_0 & 0 & o & \cdots \\ 0 & b_1 & o & \cdots \\ 0 & 0 & b_2 & \cdots \\ \vdots & \vdots & \vdots & \ddots \end{pmatrix}. \text{ Then } \begin{pmatrix}
\textbf{O} & N \\ M & \textbf{O} \\ M & N \end{pmatrix}$$ is C-image partition regular near zero over $S$.
\end{theorem}
\begin{proof}
Let $A$ be a C-set in $S$. Pick an idempotent $p$ in $J_0(S)$ with the property that $A \in p$. Let us define $B = \{x \in A : −x + A \in p\}$. Then by Lemma $4.14$ of \cite{hindman1998algebra} $B \in p$ and therefore $B$ is C-set in $S$. So pick $\overset{\rightarrow}{x} \in S^{\omega}$ in such a way that $M \overset{\rightarrow}{x} \in B^{\omega}$.

For any given $n \in \omega$, define $a_n = \sum_{t=0}^{\infty}a_{n,t} \cdot  x_t$. Then $A \cap (−a_n + A) \in p$, so pick $z_n \in A \cap (−a_n + A) \cap b_nS$ and define $y_n = z_n/b_n$. Then it follows that $$\begin{pmatrix}
\textbf{O} & N \\ M & \textbf{O} \\ M & N \end{pmatrix} \begin{pmatrix} \overset{\rightarrow}{x} \\ \overset{\rightarrow}{y} \end{pmatrix} \in C^{\omega + \omega + \omega}.$$
\end{proof}
Let us quickly recall the following definition which is Definition $4.8$ in \cite{7hind2000}.
\begin{definition}
Let $ \gamma, \delta \in \omega \cup \{\omega\}$ and let $C$ be a  matrix of order $\gamma \times \delta$ containing  finitely many nonzero entries in each row. For each $t < \delta$, let $B_t$ be (finite matrix) of dimension $u_t \times v_t$. Let $R = \{(i, j) : i < \gamma \text{ and } j \in \bigtimes_{ t < \delta}\{0, 1, \cdots , u_t - 1\}\}$. Given $t < \delta$ and
$k \in \{0, 1, \cdots , u_t - 1\}$, call $\overset{\rightarrow}{b}_{k}^{(t)}$ to be the $k$-th row of $B_t$. Then $D$ is said to be an insertion matrix of $\langle B_t \rangle_{t<\delta}$ into $C$ if and only if the rows of $D$ are all rows of the form $$c_{i,0} \cdot \overset{\rightarrow}{b}_{j(0)}^{(0)} \frown c_{i,1} \cdot \overset{\rightarrow}{b}_{j(1)}^{(1)} \frown \cdots$$ where $(i,j) \in R$.
\end{definition}
For example consider that one which is given in \cite{7hind2000}, i.e., if $C =
\begin{pmatrix}
1 & 0\\
2 & 1
\end{pmatrix}$
, $B_0 =
\begin{pmatrix}
1 & 1 \\
5 & 7
\end{pmatrix}$
, and $B_1 =
\begin{pmatrix}
0 & 1 \\ 
3 & 3 
\end{pmatrix}$
, then
$$\begin{pmatrix}
1 & 1 & 0 & 0 \\
5 & 7 & 0 & 0 \\
2 & 2 & 0 & 1 \\
2 & 2 & 3 & 3 \\
10 & 14 & 0 & 1 \\
10 & 14 & 3 & 3 \\
\end{pmatrix}$$
is an insertion matrix of $\langle B_t \rangle_{t < 2}$ into $C$.
\begin{theorem}
Let $C$ be a segmented first entries matrix. Also let $B_t$ to be a $u_t \times v_t$ (finite) image partition regular matrix  for each $t < \omega$. Then any insertion matrix of $\langle B_t \rangle_{t<\omega}$ into $C$ is C-image partition regular near zero.
\end{theorem}
\begin{proof}
Pick $A$ to be an insertion matrix of $\langle B \rangle_{t<\omega}$ into $C$. For each $t \in \omega$, pick by part \ref{thm0.06} of Theorem \ref{thm0.0}, some $m_t \in \mathbb{N}$ and a $u_t \times m_t$ first entries matrix $D_t$ with the property that for all
$\overset{\rightarrow}{y} \in \mathbb{N}^{m_t}$ there exists $ \overset{\rightarrow}{x} \mathbb{N}^{v_t}$ such that $B_t \overset{\rightarrow}{x} = D_t \overset{\rightarrow}{y}$. Let $E$ be an insertion matrix of $\langle D_t \rangle_{t<\omega}$ into $C$ in which the rows occur in the corresponding position to those of $A$. That is, whenever $i < \omega$ and $j \in \bigtimes_{t<\omega} \{0, 1, \cdots , u_t − 1\}$ and $$c_{i,0} \cdot \overset{\rightarrow}{b}_{j(0)}^{(0)} \frown c_{i,1} \cdot \overset{\rightarrow}{b}_{j(1)}^{(1)} \frown \cdots$$ is row $k$ of $A$, then $$c_{i,0} \cdot \overset{\rightarrow}{d}_{j(0)}^{(0)} \frown c_{i,1} \cdot \overset{\rightarrow}{d}_{j(1)}^{(1)} \frown \cdots$$ is row $k$ of $E$.

Let $H$ be a C-set near zero of $S$. By Lemma $4.9$ of \cite{7hind2000}, $E$ is a segmented first entries
matrix. So by Theorem \ref{thm3.3} pick $\overset{\rightarrow}{y} \in \mathbb{N}^{\omega}$ such that all entries of $E\overset{\rightarrow}{y}$ are in $H$. Let $\delta_0 =  \gamma_0 = 0$ and for $n \in \mathbb{N}$ let $\delta_n = \sum_{t=0}^{n−1} v_t$ and also let $\gamma_n = \sum_{t=0}^{n−1} m_t$. For each $n \in \omega$, choose 
$$\begin{pmatrix}
x_{\delta_n} \\ x_{\delta_{n}+1}\\ \vdots \\ x_{\delta_{n+1}-1}
\end{pmatrix} \in \mathbb{N}^{v_n} \text{ such that } B_t \begin{pmatrix}
x_{\delta_n} \\ x_{\delta_{n}+1}\\ \vdots \\ x_{\delta_{n+1}-1}
\end{pmatrix} = D_t \begin{pmatrix}
y_{\gamma_n} \\ y_{\gamma_{n}+1}\\ \vdots \\ y_{\gamma_{n+1}-1}
\end{pmatrix}.$$ Then $A \overset{\rightarrow}{x} = E \overset{\rightarrow}{y}$.
\end{proof}
As a consequence of the above theorem we have the folloing corollary:
\begin{corollary}
Let $C$ be a segmented first entries matrix and for each $t < \omega$, let $B_t$ be a $u_t \times v_t$ (finite) image partition regular matrix. Then any insertion matrix of $\langle B_t \rangle_{t<\omega}$ into $C$ is centrally image partition regular near zero.
\end{corollary}

\bibliographystyle{abbrvnat}
\bibliography{biblfile}
\end{document}